\documentclass[12pt]{amsart}

\usepackage{fullpage}
\usepackage[all]{xy}
\usepackage{latexsym,amsthm,amssymb,amsfonts,amsmath,mathrsfs}
\usepackage{multienum,color,tabls,listliketab,mathrsfs,extarrows,multicol,indentfirst,array}
\usepackage{fullpage, enumitem}
\usepackage{pifont,graphicx,txfonts}

\def\CP{{\mathbb C \mathbb P}}
\def\C{{\mathbb C}}

\def\T{{\mathbb T}}

\def\J{{\mathcal J}}
\def\R{\mathbb{R}}
\def\C{\mathbb{C}}
\def\G{\mathbb{G}}

\def\N{\mathbb{N}}
\def\T{\mathbb{T}}
\def\TM{\mathbb{T}M}
\def\Z{\mathbb{Z}}

\def\top{\mbox{top}}

\def\Im{{\rm Im\/}}

\newtheorem{theorem}{Theorem}[section]
\newtheorem{lemma}[theorem]{Lemma}

\newtheorem{definition}[theorem]{Definition}
\newtheorem{example}[theorem]{Example}

\newtheorem{proposition}[theorem]{Proposition}
\newtheorem{corollary}[theorem]{Corollary}

\newenvironment{acknowledgement}[1][Acknowledgements]
{\begin{trivlist} \item[\hskip \labelsep {\bfseries #1}]}
{\end{trivlist}}

\newcommand{\pd}[2]{\dfrac{\partial#1}{\partial #2}}
\newcommand{\dif}[1]{{\rm d}#1}
\newcommand{\cyc}[1]{\left<#1\right>}

\newcommand{\Set}[2]{\left\{#1\bigg\vert #2\right\}}

\title{Aeppli and Bott-Chern cohomology for bi-generalized Hermitian manifolds and $d'd''$-lemma}

\author{Tai-Wei Chen}
\address{Mathematics Department\\  National Tsing Hua University\\ Hsinchu, Taiwan}
\email{d937203@oz.nthu.edu.tw}

\author{Chung-I Ho}
\address{Mathematics Department\\  National Tsing Hua University\\  National Center of Theoretical Sciences\\ Mathematical Division\\ Hsinchu, Taiwan}
\email{ciho@math.cts.nthu.edu.tw}

\author{Jyh-Haur Teh}
\address{Mathematics Department\\  National Tsing Hua University\\ Hsinchu, Taiwan}
\email{jyhhaur@math.nthu.edu.tw}
\date{}
\begin{document}

\begin{abstract}
We define Aeppli and Bott-Chern cohomology for bi-generalized complex manifolds and show that they are finite dimensional for compact bi-generalized Hermitian manifolds. For totally bounded double complexes $(A, d', d'')$, we show that
the validity of $d'd''$-lemma is equivalent to having the same dimension of several cohomology groups. Some calculations
of Bott-Chern cohomology groups of some bi-generalized Hermitian manifolds are given.
\end{abstract}


\maketitle

\section{Introduction}
Generalized complex geometry is a framework that unifies complex and symplectic geometry. This theory was proposed by Hitchin in \cite{H}, and further developed by his students Gualtieri and Cavancanti \cite{G1, C}.
String theorists are interested in this theory as it arises naturally in compactifying type II theories. As indicated in \cite{GMPT}, on a six dimensional internal manifold $M$, the structure group
of $T(M)\oplus T^*(M)$ being $SU(3)\times SU(3)$ implies the existence of two compatible generalized almost complex structures $(\J_1, \J_2)$ on $M$ of which $\J_1$ is integrable while the integrability of $\J_2$
fails in the presence of a RR flux. Tseng-Yau \cite{TY1} mimicked the case of solving Maxwell equations on some four-manifolds showing that the Bott-Chern and Aeppli cohomology for generalized complex manifolds
can be used to count massless fields for a general supersymmetric Minkowski type II compactification with RR flux. To avoid too much technical difficulties, the natural first step towards an understanding of the geometry
of $M$ is not to consider the presence of RR fluxes. In this case we have two integrable generalized complex structures $\J_1, \J_2$ and the usual exterior derivative $\dif$ has a decomposition $\dif=\delta_++\delta_-+\overline{\delta}_++\overline{\delta}_-$. We are particular interested in the Bott-Chern and Aeppli cohomology defined by $\delta_+$ and $\delta_-$.

To build a general mathematical framework for manifolds with two generalized complex structures, we consider bi-generalized Hermitian manifolds $(M, \J_1, \J_2, \G)$ where $\J_1, \J_2$ are compatible generalized complex structures and $\G$ is a generalized metric that commutes with $\J_1, \J_2$. As a special case when $\G=-\J_1\J_2$, $M$ becomes a generalized K\"ahler manifold (\cite{G1, G2}).
The first important problem is the finiteness of dimensions of the Bott-Chern and Aeppli cohomology groups of $M$. This is proved by using elliptic operator theory and methods from Schweitzer's (\cite{S}).

Bott-Chern cohomology of complex manifolds plays an important role especially when the manifolds do not satisfy the $\partial\overline{\partial}$-lemma. In this case, Bott-Chern cohomology may be different
from Dolbeault cohomology. As pointed out by Tseng and Yau in \cite{TY2}, Bott-Chern cohomology gives more information for non-K\"ahler manifolds that do not satisfy $\partial\overline{\partial}$-lemma.
Hence it is important to understand $\delta_+\delta_-$-lemma and its consequences. On a complex manifold $M$, the Hodge-de Rham spectral sequence $E_*^{*,*}$ is built from the double complex $(\Omega^*(M),\partial,\bar{\partial})$ which relates the Dolbeault cohomology of $M$ to the de Rham cohomology of $M$. It is well known that $E^{p,q}_1$ is isomorphic to $H^p(M,\Omega^q)$ and the spectral sequence $E^{*,*}_r$ converges to $H^*(M,\C)$.
We first give a purely algebraic description of the $\partial\overline{\partial}$-lemma of a double complex in the frame of spectral sequences and then show that $\partial\overline{\partial}$-lemma is equivalent to the equalities
of several cohomologies, in particular, Bott-Chern cohomology and
Dolbeault cohomology when they are finite dimensional. This largely generalizes the result obtained by Angella and Tomassini in \cite{AT}.

This paper is organized as follows: In section 2, we show that on a compact bi-generalized Hermitian manifold, several Bott-Chern and Aeppli cohomologies are finite dimensional by using the theory of
elliptic operators and prove a Serre duality for Bott-Chern cohomologies of compact generalized K\"ahler manifolds. In section 3, we show that the $\dif'\dif''$-lemma condition is equivalent to having the same dimension of several cohomology groups when one of these groups is finite dimensional, in particular, the equivalence between Bott-Chern cohomology and de Rham cohomology. Applying this result to compact generalized K\"ahler manifolds, we are able to show that our generalized Bott-Chern cohomology groups are isomorphic to
$\delta_+$-cohomology and $\delta_-$-cohomology groups.
In section 4, we compute the generalized Bott-Chern cohomology groups of $\R^2, \R^4, \T^2, \T^4$ with generalized complex structures induced from non-compatible complex and symplectic structures.

\begin{acknowledgement}
The authors  thanks the Taiwan National Center for
Theoretical Sciences (Hsinchu) for providing a wonderful working
environment. The third author thanks Gil Cavalcanti for pointing him to reference
\cite{BCG}.
\end{acknowledgement}

\section{Bott-Chern and Aeppli cohomology for bi-generalized complex manifolds}
We refer the reader to \cite{G1, Ca07, H} for the basic of generalized complex geometry. We give a brief recall of some terminologies here.
On a smooth manifold $M$, a generalized metric is an orthogonal, self-adjoint operator $\G:\TM \rightarrow \TM$ on the generalized tangent space $\TM:=TM\oplus T^*M$ such that
$\cyc{\G e, e}>0$ for $e\in \TM\backslash \{0\}$ where $\cyc{\ , }$ is the natural pairing defined by $\cyc{X+\xi, Y+\eta}=\frac{1}{2}(\xi(Y)+\eta(X))$ for $X, Y\in TM, \xi, \eta\in T^*M$.

\begin{definition}
A bi-generalized complex structure on a smooth manifold $M$ is a pair $(\J_1, \J_2)$ where $\J_1, \J_2$ are commuting generalized complex structures on $M$. A
bi-generalized complex manifold is a smooth manifold $M$ with a bi-generalized complex structure. A bi-generalized Hermitian manifold $(M, \J_1, \J_2, \G)$
is an oriented bi-generalized complex manifold $(M, \J_1, \J_2)$ with a generalized metric $\G$ which commutes with $\J_1$ and $\J_2$.
\end{definition}

In particular, any generalized K\"ahler manifold is also a bi-generalized Hermitian manifold.

\begin{lemma}\label{decompose}
Let $V$ be a vector space, finite or infinite dimensional over $\C$, and $L_1, L_2:V \rightarrow V$ be two linear transformations. Suppose that $V=\bigoplus^n_{p=1}V^p_1=\bigoplus^m_{q=1}V^q_2$ where
$V^p_1, V^q_2$ are some eigenspaces of $L_1$ and $L_2$ respectively. If $L_1$ and $L_2$ commute, then
$$V=\bigoplus^{n, m}_{p=1, q=1} V^{p, q}$$ where $V^{p,q}=V^p_1\cap V^q_2$.
\end{lemma}

\begin{proof}
Let $\alpha_p, \beta_q$ be eigenvalues of $L_1, L_2$ with eigenspace $V^p_1, V_2^q$ respectively. For $u\in V^p_1$, $L_1(L_2(u))=L_2(L_1(u))=L_2(\alpha_p u)=\alpha_p L_2(u)$ which implies
that $L_2(u)\in V^p_1$. If $u=\sum^m_{q=1}v_q$ where $v_q\in V^q_2$, then we have $L^r_2u=\sum^m_{q=1}\beta^r_qv_q$ for any nonnegative integer $r$. Let
$A=[a_{ij}]$ be the Vandermonde matrix where $a_{ij}=\beta_j^{i-1}$ for $1\leq i, j\leq m$. We see that $[v_1 \ \ v_2 \ \  \cdots \ \ v_m]^t=A^{-1}[u \
\ L_2u \ \ \cdots \ \ L_2^{m-1}u]^t$ and hence each $v_q\in V^p_1$. This implies that $V^p_1=\bigoplus^m_{q=1}V^p_1\cap V^q_2$.
\end{proof}

\begin{definition}
Given a bi-generalized complex manifold $(M, \J_1, \J_2)$, we define
$$U^{p, q}:=U^p_1\cap U^q_2$$ where $U^p_1, U^q_2\subset \Gamma(\Lambda^*\TM\otimes \C)$ are eigenspaces of $\J_1, \J_2$ associated to the eigenvalues $ip$ and $iq$ respectively. By the lemma above and the fact that $\J_1, \J_2$ are generalized complex structures, the exterior derivative $\dif$ is an operator from
$U^{p, q}$ to $U^{p+1, q+1}\oplus U^{p+1, q-1}\oplus U^{p-1, q+1} \oplus U^{p-1, q-1}$. We write
$$\delta_+:U^{p, q} \rightarrow U^{p+1, q+1}, \delta_-:U^{p, q} \rightarrow U^{p+1, q-1}$$ and
$$\overline{\delta}_+:U^{p, q} \rightarrow U^{p-1, q-1}, \overline{\delta}_-:U^{p, q} \rightarrow U^{p-1, q+1}$$
by projecting $\dif$ into the corresponding spaces.
\end{definition}

Recall that for a generalized complex manifold $(M, \J)$, there is a decomposition $\dif=\partial+\overline{\partial}$ where $\partial:U^p \rightarrow U^{p+1}, \overline{\partial}:U^p \rightarrow U^{p-1}$
are projections of $\dif$. On a bi-generalized complex manifold $(M, \J_1, \J_2)$, we have a finer decomposition $\partial_1=\delta_++\delta_-, \overline{\partial}_1=\overline{\delta}_++\overline{\delta}_-, \partial_2=\delta_++\overline{\delta}_-,
\overline{\partial}_2=\overline{\delta}_++\delta_-$ where $(\partial_1, \overline{\partial}_1), (\partial_2, \overline{\partial}_2)$ are the decompositions of $\dif$ with respect to $\J_1, \J_2$ respectively.

\begin{definition}
Let $M$ be a real manifold of dimension $2n$ and $\alpha=\sum^{2n}_{a=0}\alpha^a, \beta=\sum^{2n}_{b=0}\beta^b$ be two complex forms on $M$ where $\alpha^a, \beta^b$
are degree $a$ and degree $b$ components of $\alpha$ and $\beta$ respectively. Define
$$\sigma(\alpha^a)=\left\{
                     \begin{array}{ll}
                       (-1)^{\frac{a}{2}}\alpha^a, & \hbox{ if } a \mbox{ is even}\\
                       (-1)^{\frac{a-1}{2}}\alpha^a, & \hbox{ if } a \mbox{ is odd}
                     \end{array}
                   \right.
$$
and write $(\alpha)_{\top}$ for the degree $2n$ component of $\alpha$. The Chevalley pairing is defined to be
$$(\alpha, \beta)_{Ch}:=-\sum^n_{j=0}(-1)^j\left(\alpha^{2j}\wedge \beta^{2n-2j}+\alpha^{2j+1}\wedge \beta^{2n-2j-1}\right)$$
\end{definition}
We have $(\alpha, \beta)_{Ch}=-(\sigma(\alpha)\wedge \beta)_{\top}$. The following result is a direct calculation.

\begin{lemma}\label{exact}
For any two complex forms $\alpha, \beta$ on a generalized complex manifold $(M, \J)$ of real dimension $2n$, we have
$$(\dif\alpha, \beta)_{Ch}+(\alpha, d\beta)_{Ch}=(\dif(-\widetilde{\sigma}(\alpha)\wedge \beta))_{\top}$$ where
$$\widetilde{\sigma}(\alpha^a)=\left\{
                     \begin{array}{ll}
                       (-1)^{\frac{a}{2}}\alpha^a, & \hbox{ if } a \mbox{ is even}\\
                       (-1)^{\frac{a+1}{2}}\alpha^a, & \hbox{ if } a \mbox{ is odd}
                     \end{array}
                   \right.
$$
\end{lemma}

Recall that a generalized Hermitian manifold $(M, \J, \G)$ is an oriented generalized complex manifold $(M, \J)$ with a compatible generalized metric $\G$ (see \cite{BCG, Ca07}).
The generalized tangent space $\TM$ is split into $\pm1$ eigenbundles $V_{\pm}$ of $\G$. The orientation of $M$ induces an orientation of $V_+$. For $x\in M$ and a positive normal basis $\{e_1, \cdots, e_{2n}\}$ of $V_{+, x}$, let
$$\star:=-e_{2n}\cdots e_1\in CL(\T_xM)$$
The generalized Hodge star operator $\star:\Lambda^{\bullet}T^*_xM \rightarrow \Lambda^{\bullet}T^*_xM$ is defined by the Clifford action
$$\star \alpha:=\star\cdot \alpha$$
on spinors. This can be extended to a $\C$-linear map $\star:\Lambda^{\bullet}T^*M\otimes \C \rightarrow \Lambda^{\bullet}T^*M\otimes \C$. We write
$$\overline{\star}\alpha:=\star\overline{\alpha}$$
where $\overline{\alpha}$ is the complex conjugation of $\alpha\in \Lambda^{\bullet}T^*M\otimes \C$.

For $p\in \Z$ and $\alpha, \beta\in U^p$, the generalized Hodge inner product is defined to be
$$h(\alpha, \beta):=\int_M(\alpha, \overline{\star}\beta)_{Ch}$$ which is a positive definite symmetric bilinear form on $U^p$.
If without mentioned explicitly, we write $\delta^*$ for the $h$-adjoint of an operator $\delta$.

\begin{proposition}\label{adjoint}
Let $M$ be a compact generalized Hermitian manifold. Then
$\partial^*=-\overline{\star}^{-1}\partial \overline{\star}$ and $\overline{\partial}^*=-\overline{\star}^{-1}\overline{\partial}\overline{\star}$.
\end{proposition}

\begin{proof}
By Lemma \ref{exact}, we have $(\dif\alpha, \beta)_{Ch}+(\alpha, \dif\beta)_{Ch}=(\dif(-\widetilde{\sigma}(\alpha)\wedge \beta))_{\top}$ for any two complex forms $\alpha, \beta$ on $M$.
For $\alpha\in U^{k-1}, \beta\in U^{-k}$, by comparing degrees, we see that both $(\overline{\partial}\alpha, \beta)_{Ch}$ and $(\alpha, \overline{\partial} \beta)_{Ch}$ vanish and hence $(d(-\widetilde{\sigma}(\alpha)\wedge \beta))_{\mbox{top}}=(\partial \alpha, \beta)_{Ch}+(\alpha, \partial \beta)_{Ch}$. Therefore, by Stokes theorem,
$$h(\partial \alpha, \beta)=\int_M(\partial \alpha, \overline{\star}\beta)_{Ch}=-\int_M(\alpha, \partial\overline{\star}\beta)_{Ch}=h(\alpha, -\overline{\star}^{-1}\partial \overline{\star} \beta)$$
which implies that $\partial^*=-\overline{\star}\partial \overline{\star}$. Similarly we have the result for $\overline{\partial}^*$.
\end{proof}

\begin{proposition}\label{adjoint and commute}
Let $M$ be a compact bi-generalized Hermitian manifold. Then
\begin{enumerate}
\item $\overline{U^{p, q}}=U^{-p, -q}$ for any $p, q$.

\item $\delta_+^*=-\overline{\star}^{-1}\delta_+ \overline{\star}$ and $\delta_-^*=-\overline{\star}^{-1}\delta_-\overline{\star}$.

\item $\overline{\delta_+\alpha}=\overline{\delta}_+(\overline{\alpha}), \overline{\delta_-\alpha}=\overline{\delta}_-(\overline{\alpha}),
\overline{\overline{\delta}_+\alpha}=\delta_+(\overline{\alpha}), \overline{\overline{\delta}_-\alpha}=\delta_-(\overline{\alpha})$.

\item $\delta_+\delta_-=-\delta_-\delta_+, \delta_+\overline{\delta}_-=-\overline{\delta}_-\delta_+,
\overline{\delta}_+\overline{\delta}_-=-\overline{\delta}_-\overline{\delta}_+, \overline{\delta}_+\delta_-=-\delta_-\overline{\delta}_+$.

\item $\delta_+\overline{\delta}_++\overline{\delta}_+\delta_++\delta_-\overline{\delta}_-+\overline{\delta}_-\delta_-=0$.
\end{enumerate}
\end{proposition}

\begin{proof}
\begin{enumerate}
\item This follows from $\overline{U^p_1}=U^{-p}_1$, $\overline{U^q_2}=U^{-q}_2$.

\item Note that by Proposition~\ref{adjoint},
$\partial^*_1=-\overline{\star}^{-1}\partial_1\overline{\star}=-\overline{\star}^{-1}(\delta_++\delta_-)\overline{\star}=-\overline{\star}^{-1}\delta_+\overline{\star}-\overline{\star}^{-1}\delta_-\overline{\star}$, and
from $\partial_1=\delta_++\delta_-$, we have $\partial^*_1=\delta_+^*+\delta_-^*$. Comparing the degrees of both sides of $\partial_1^*\alpha$ for $\alpha\in U^{p, q}$, we get the result.

\item Since $\dif$ is a real operator, this means that $\overline{d\alpha}=d\overline{\alpha}$.
Hence we have $\overline{\delta_+\alpha}+\overline{\delta_-\alpha}+\overline{\overline{\delta}_+\alpha}+\overline{\overline{\delta}_-\alpha}
=\delta_+\overline{\alpha}+\delta_-\overline{\alpha}+\overline{\delta}_+\overline{\alpha}+\overline{\delta}_-\overline{\alpha}$. By comparing degrees of both sides, we get the result.

\item Since $d=\partial_1+\overline{\partial}_1=\partial_2+\overline{\partial}_2$ and $\partial_1=\delta_++\delta_-, \overline{\partial}_1=\overline{\delta}_++\overline{\delta}_-$, $\partial_2=\delta_++\overline{\delta}_-$,
$\overline{\partial}_2=\delta_-+\overline{\delta}_+$, the result follows from the fact that the square of $\partial_1, \overline{\partial}_1, \partial_2, \overline{\partial}_2$ are 0.

\item Since $\dif^2=(\delta_++\delta_-+\overline{\delta}_++\overline{\delta}_-)^2=0$, expand it and use the anti-commutative relations above, we get the result.
\end{enumerate}
\end{proof}

In the following, we work on a compact bi-generalized Hermitian manifold $(M, \J_1, \J_2, \G)$.

\begin{proposition}\label{elliptic complexes}
On a generalized complex manifold $M$ of real dimension $2n$, the following complexes are elliptic:
\begin{enumerate}
\item $$0 \rightarrow \bigoplus_{k \mbox{ even}}\Gamma(U^k) \overset{d}{\longrightarrow}\bigoplus_{k \mbox{ odd}}\Gamma(U^k) \overset{d}{\longrightarrow} \bigoplus_{k \mbox{ even}}\Gamma(U^k) \rightarrow 0$$
\item $$\cdots \rightarrow \Gamma(U)^{k-1} \overset{\partial}{\longrightarrow} \Gamma(U^k) \overset{\partial}{\longrightarrow} \Gamma(U^{k+1}) \rightarrow \cdots$$
\item $$\cdots \rightarrow \Gamma(U)^{k-1} \overset{\overline{\partial}}{\longrightarrow} \Gamma(U^k) \overset{\overline{\partial}}{\longrightarrow} \Gamma(U^{k+1}) \rightarrow \cdots$$
\end{enumerate}
On a bi-generalized complex manifold $M$, for $j=1, 2, 3, 4$, the following complexes are elliptic:
$$\cdots \rightarrow \Gamma(U^{p, q}) \overset{\delta_j}{\longrightarrow} \Gamma(U^{(p, q)_j}) \overset{\delta_j}{\longrightarrow} \Gamma(U^{(p, q)_{jj}}) \rightarrow \cdots$$
where $U^{(p, q)_j}$ is the codomain of $\delta_j$ from $U^{p, q}$, and $U^{(p, q)_{jj}}$ is the codomain of $\delta_j$ from $U^{(p, q)_j}$.
\end{proposition}

\begin{proof}
\begin{enumerate}
\item The principal symbol of $\dif$ is $\sigma(\dif)(x, \xi)=\xi\wedge\cdot$ for $x\in M, \xi\in T^*_xM\backslash \{0\}$. From the fact that $\xi\wedge\alpha=0$ implies $\alpha=\xi\wedge \beta$ for some $\beta\in \Lambda^{\bullet}T^*_xM$,
we see that the sequence
$$0 \rightarrow \bigoplus_{k \mbox{ even}}U^k_x \overset{\sigma(\dif)(x, \xi)}{\longrightarrow} \bigoplus_{k \mbox{ odd}} U^k_x \overset{\sigma(\dif)(x, \xi)}{\longrightarrow} \bigoplus_{k \mbox{ even}}U^k_x \rightarrow 0$$
is exact. Hence $\dif$ is elliptic.

\item Let $L$ be the Clifford annihilator of a pure form $\rho\in U^n_x$.
For $\xi\in T^*_xM\backslash \{0\}$, we may write $\xi=\xi_1+\overline{\xi_1}$ where $\xi\in L, \overline{\xi}\in \overline{L}$. Then the principal symbols of
$\partial$ and $\overline{\partial}$ are $\sigma(\partial)(x, \xi)=\xi_1\bullet $ and $\sigma(\overline{\partial})(x, \xi)=\overline{\xi_1}\bullet$ respectively where $\xi_1\bullet$ and $\overline{\xi_1}\bullet$ are
the Clifford action by $\xi_1$ and $\overline{\xi_1}$ on $U^k_x$ respectively. Note that $\overline{\rho}$ is a Clifford annihilator of $\overline{L}$ and for $\alpha\in U^k_x$, $\alpha=\eta\bullet \overline{\rho}$ for some $\eta\in \Lambda^{n+k}L$. If $\xi\bullet \alpha=0$,
$\xi\bullet(\eta\bullet \overline{\rho})=(\xi\wedge \eta)\bullet \overline{\rho}=0$. Since the action of $\Lambda^{\bullet} L$ on $\Lambda^{\bullet}T^*_xM\otimes \C$ is faithful, we have $\xi\wedge \eta=0$.
So $\eta=\xi\wedge \eta'$ for some $\eta'$. This implies that the sequence of the principal symbol is exact hence the complex is elliptic.

\item This is similar to 2.
\end{enumerate}

On a bi-generalized complex manifold $M$, if $\xi\in \Lambda^{\bullet}T^*_xM\subset L_1\oplus \overline{L_1}$, we may write $\xi=\xi_1+\overline{\xi_1}$. Since
$\xi_1\in L_1\subset L_2\oplus \overline{L_2}$, we further write $\xi_1=\xi_{11}+\xi_{12}$. Then $\overline{\xi_1}=\overline{\xi_{11}}+\overline{\xi_{12}}$.
We get the principal symbols
$$\sigma(\delta_+)(x, \xi)=\xi_{11}\bullet, \ \sigma(\delta_-)(x, \xi)=\xi_{12}\bullet, \ \sigma(\overline{\delta_+})(x, \xi)=\overline{\xi_{11}}\bullet,\ \sigma(\overline{\delta_-})(x, \xi)=\overline{\xi_{12}}\bullet$$
Argument similar as in 2 implies that the complexes for $\delta_j, j=1, 2, 3, 4$ are elliptic.
\end{proof}

The following result is probably well known, but since no reference for it is found, we give a proof here.
\begin{lemma}\label{delta+delta^*}
Let $E_j \rightarrow M$ be Hermitian vector bundles on $M$ for $j=1, ..., N+1$. Suppose that
$$\cdots \rightarrow \Gamma(E_j) \overset{\delta_j}{\longrightarrow} \Gamma(E_{j+1}) \overset{\delta_{j+1}}{\longrightarrow} \Gamma(E_{j+2}) \rightarrow \cdots$$
is an elliptic complex where each $\delta_j:\Gamma(E_j) \rightarrow \Gamma(E_{j+1})$ is a differential operator of order $k\in \N$.
Let $\delta_j^*:\Gamma(E_{j+1}) \rightarrow \Gamma(E_j)$ be the adjoint of $\delta_j$, i.e.,
$\cyc{\delta_j\alpha, \beta}_{j+1}=\cyc{\alpha, \delta_j^*\beta}_j$ for $\alpha\in \Gamma(E_j), \beta\in \Gamma(E_{j+1})$, then
the operator $\delta:\Gamma(E_j) \rightarrow \Gamma(E_{j+1}) \oplus \Gamma(E_{j-1})$ defined by
$$\dif_j:=\delta_j+\delta^*_j$$ is an elliptic operator for $j=1, ..., N$.
\end{lemma}

\begin{proof}
Let $x\in M, \xi\in T^*_xM\backslash \{0\}$. The principal symbol
$$\sigma(\dif_j)(x, \xi)=\sigma(\delta_j)(x, \xi)+(-1)^k\sigma(\delta_j)(x, \xi)^*$$
where $\sigma(\delta_j)(x, \xi)^*$ is the transpose
of $\sigma(\delta_j)(x, \xi)$. Suppose that $\sigma(\dif_j)(x, \xi)\alpha=0$, then
$$\sigma(\delta_j)(x, \xi)^*\sigma(\delta_j)(x, \xi)\alpha=0$$
which implies that
$$\cyc{\sigma(\delta_j)(x, \xi)\alpha, \sigma(\delta_j)(x, \xi)\alpha}=\cyc{\alpha, \sigma(\delta_j)(x, \xi)^*\sigma(\delta_j)(x, \xi)\alpha}=0$$
So $\sigma(\delta_j)(x,\xi)\alpha=0=\sigma(\delta_j)(x,\xi)^*\alpha$.
Since the complex is elliptic, there exists $\beta\in \Gamma(E_{j-1})$ such that $\alpha=\sigma(\delta_{j-1})(x, \xi)\beta$. Then $\sigma(\delta_j)(x, \xi)^*\sigma(\delta_{j-1})(x, \xi)\beta=0$,
and we have again
$$\cyc{\sigma(\delta_{j-1})(x, \xi)\beta, \sigma(\delta_{j-1})(x, \xi)\beta}=\cyc{\beta, \sigma(\delta_j)(x, \xi)^*\sigma(\delta_{j-1})(x, \xi)(\beta)}=0$$
Hence $\alpha=\sigma(\delta_{j-1})(x, \xi)\beta=0$.
Therefore $d_j$ is an elliptic operator.
\end{proof}

\begin{definition} For notation simplification, we denote $\delta_1=\delta_+, \delta_2=\delta_-, \delta_3=\overline{\delta}_-, \delta_4=\overline{\delta}_+$.
For $(i, j)=(1, 2), (1, 3), (2, 4), (3, 4)$, define
$$\Delta_{i, j}:=\delta_i\delta_j\delta^*_j\delta^*_i+\delta_j^*\delta^*_i\delta_i\delta_j+\delta^*_j\delta_i\delta^*_i\delta_j+\delta^*_i\delta_j\delta^*_j\delta_i
+\delta^*_i\delta_i+\delta^*_j\delta_j$$

Let $\Delta^{p, q}_{i, j}$ be the restriction of the operator $\Delta_{i, j}$ to $U^{p, q}$.
For an operator $\delta$, we write
$$\Delta_{\delta}:=\delta\delta^*+\delta^*\delta.$$
for its Laplacian.
\end{definition}

The following result is obtained by some calculation from Proposition \ref{adjoint and commute}.
\begin{corollary}\label{Kahler identity}
\begin{enumerate}
\item If $i<j$, $\delta^*_i\delta_j=-\delta_j\delta^*_i$, $\delta_i\delta^*_j=-\delta^*_j\delta_i$.
\item If $(i, j)=(1, 2), (1, 3), (2, 4), (3, 4)$, we have $\delta^*_i\delta^*_j=-\delta^*_j\delta^*_i$.
\end{enumerate}
\end{corollary}

The following result is a combination of results and methods in \cite{Ca07} and \cite{S}.
\begin{theorem}\label{elliptic}
Given a compact bi-generalized Hermitian manifold $(M, \J_1, \J_2, \G)$. Let $(i, j)=(1, 2), (1, 3), (2, 4), (3, 4)$. Then
\begin{enumerate}
\item The operator $\Delta_{i, j}^{p, q}$ is elliptic and self-adjoint;
\item $\mbox{Ker} \delta^{p, q}_i\cap \mbox{Ker} \delta^{p, q}_j=\mathscr{H}^{p, q}_{\Delta_{i, j}}\bigoplus (\Im\delta_i\delta_j\cap U^{p, q})$ where $\mathscr{H}^{p, q}_{\Delta_{i,
j}}=\mbox{Ker }\Delta^{p, q}_{i, j}$ is the space of harmonic forms with respect to $\Delta^{p, q}_{i, j}$.
\end{enumerate}
\end{theorem}

\begin{proof}
\begin{enumerate}
\item
Denote $\widetilde{\Delta}_{i, j}$ for the
highest order terms of $\Delta_{i, j}$. By Proposition \ref{adjoint
and commute}, we have
\begin{align*}
\Delta_{\delta_i}\Delta_{\delta_j}&=(\delta_i\delta^*_i+\delta^*_i\delta_i)(\delta_j\delta^*_j+\delta^*_j\delta_j)\\
&=\delta_i\delta^*_i\delta_j\delta^*_j+\delta_i\delta^*_i\delta^*_j\delta_j+\delta^*_i\delta_i\delta_j\delta^*_j+\delta^*_i\delta_i\delta^*_j\delta_j\\
&=\delta_i\delta_j\delta^*_j\delta^*_i+\delta^*_j\delta_i\delta^*_i\delta_j+\delta^*_i\delta_j\delta^*_j\delta_i+\delta^*_j\delta^*_i\delta_i\delta_j=\widetilde{\Delta}_{i, j}
\end{align*}
Since $\Delta_{\delta_i}=(\delta_i+\delta_i^*)^2$ and $\Delta_{\delta_j}=(\delta_j+\delta_j^*)^2$ are elliptic, by the multiplicative property of
principal symbols, $\widetilde{\Delta}_{i, j}$ is elliptic and hence $\Delta_{i, j}$ is elliptic. The self-adjointness is clear.

\item First we note that $\Delta_{i, j}H=0$ if and only if $(\delta_i\delta_j)^*H=\delta_iH=\delta_jH=0$. By the theory of elliptic operators(\cite[Theorem 4.12]{W}), for $u\in U^{p, q}$,
there is a Hodge decomposition $u=H+\delta_i\delta_jv+\delta^*_iw_1+\delta^*_jw_2$ for
some $H\in \mathscr{H}_{\Delta_{i, j}}$ and some forms $v, w_1, w_2$. Then $u\in (\mbox{Ker } \delta^{p, q}_i)\cap (\mbox{Ker } \delta^{p, q}_j)$ if
and only if $\delta_i(\delta^*_iw_1+\delta^*_jw_2)=\delta_j(\delta^*_iw_1+\delta^*_jw_2)=0$. This is equivalent to
$$h(\delta^*_iw_1+\delta^*_jw_2,
\delta^*_iw_1+\delta^*_jw_2)=h(w_1, \delta_i(\delta^*_iw_1+\delta^*_jw_2))+h(w_2, \delta_j(\delta^*_iw_1+\delta^*_jw_2))=0$$
which is also equivalent to $\delta^*_iw_1+\delta^*_jw_2=0$. Therefore we have $\mbox{Ker} \delta^{p, q}_i\cap
\mbox{Ker} \delta^{p, q}_j=\mathscr{H}^{p, q}_{\Delta_{i, j}}\bigoplus (\Im\delta_i\delta_j\cap U^{p, q})$.
\end{enumerate}
\end{proof}

\begin{definition}
For a bi-generalized complex manifold $(M, \J_1, \J_2)$, we define the Bott-Chern cohomology groups on
$M$ to be
$$H^{p, q}_{BC, \delta_+\delta_-}(M):=\frac{\mbox{Ker } \delta_+^{p, q}\cap \mbox{Ker } \delta^{p, q}_-}{\delta_+\delta_-U^{p-2,
q}}, H^{p, q}_{BC, \delta_+\overline{\delta}_-}(M):=\frac{\mbox{Ker
} \delta_+^{p, q}\cap \mbox{Ker } \overline{\delta}^{p,
q}_-}{\delta_+\overline{\delta}_-U^{p, q-2}},$$

$$H^{p, q}_{BC, \overline{\delta}_+\delta_-}(M):=\frac{\mbox{Ker } \overline{\delta}^{p, q}_+\cap \mbox{Ker } \delta_-^{p, q}}{\overline{\delta}_+\delta_-U^{p, q+2}},
H^{p, q}_{BC,
\overline{\delta}_+\overline{\delta}_-}(M):=\frac{
\mbox{Ker } \overline{\delta}^{p, q}_+\cap \mbox{Ker }\overline{\delta}_-^{p, q}}{\overline{\delta}_+\overline{\delta}_-U^{p+2, q}}.$$

\end{definition}

By Theorem \ref{elliptic}, we know that $H^{p, q}_{BC, \delta_i\delta_j}(M)\cong \mathscr{H}^{p, q}_{\Delta_{i, j}}(M)$ and from the theory of elliptic operators, we know that the spaces of harmonic forms are of finite dimension. Hence we have the
following fundamental result.

\begin{theorem}\label{BCfinite}
On a compact bi-generalized Hermitian manifold $M$, all these Bott-Chern cohomology groups are finite
dimensional over $\C$.
\end{theorem}

Similarly we may define Aeppli cohomology groups for
bi-generalized complex manifolds as follows.

\begin{definition}
For a bi-generalized complex manifold $(M, \J_1, \J_2)$, we define the
Aeppli cohomology groups on $M$ to be
$$H^{p, q}_{A, \delta_+\delta_-}(M):=\frac{\mbox{Ker } \delta_+^{p+1, q-1} \delta^{p, q}_-}{\Im\delta_+^{p-1, q-1}+\Im\delta_-^{p-1, q+1}},
H^{p, q}_{A, \delta_+\overline{\delta}_-}(M):=\frac{\mbox{Ker} \delta_+^{p-1, q+1}\overline{\delta}^{p, q}_-}{\Im\delta_+^{p-1, q-1}+\Im\overline{\delta}_-^{p+1, q-1}},$$

$$H^{p, q}_{A, \overline{\delta}_+\delta_-}(M):=\frac{\mbox{Ker }\overline{\delta}^{p+1, q-1}_+ \delta_-^{p, q}}{\Im\overline{\delta}_+^{p+1, q+1}+\Im\delta_-^{p-1, q+1}},
H^{p, q}_{A,\overline{\delta}_+\overline{\delta}_-}(M):=\frac{\mbox{Ker }
\overline{\delta}_+^{p-1, q+1}\overline{\delta}^{p,
q}_-}{\Im\overline{\delta}_+^{p+1, q+1}+\Im\overline{\delta}_-^{p+1, q-1}}.$$
\end{definition}

Similar to the case of Bott-Chern cohomology groups, we have the following finiteness result.

\begin{theorem}
On a compact bi-generalized Hermitian manifold $M$, all these Aeppli cohomology groups are finite
dimensional over $\C$.
\end{theorem}

\begin{proposition}\label{conjugation}
On a bi-generalized complex manifold $M$, complex conjugation induces the following isomorphisms:
$$H^{p, q}_{BC, \delta_+\delta_-}(M)\cong H^{-p, -q}_{BC, \overline{\delta}_+\overline{\delta}_-}(M),
H^{p, q}_{BC, \overline{\delta}_+\delta_-}(M)\cong
H^{-p, -q}_{BC, \delta_+\overline{\delta}_-}(M)$$

$$H^{p, q}_{A, \delta_+\delta_-}(M)\cong H^{-p, -q}_{A, \overline{\delta}_+\overline{\delta}_-}(M),
H^{p, q}_{A, \overline{\delta}_+\delta_-}(M)\cong
H^{-p, -q}_{A, \delta_+\overline{\delta}_-}(M)$$
\end{proposition}

We write $H^*_{\partial}(M), H^*_{\overline{\partial}}(M), H^{*, *}_{\delta_+}(M), H^{*, *}_{\delta_-}(M)$ for the cohomology groups defined by the corresponding operators.
By the ellipticity of $\Delta_{\partial}$ and $\Delta_{\overline{\partial}}$, we see that on a compact generalized Hermitian manifold $M$, the cohomology groups $H^k_{\partial}(M)$ and $H^{k}_{\overline{\partial}}(M)$ are finite dimensional over $\C$ for any $k\in \Z$, and on a compact bi-generalized Hermitian manifold, $H^{p, q}_{\delta_+}(M)$ and $H^{p, q}_{\delta_-}(M)$ are finite dimensional over $\C$ for any $p, q\in \Z$.

\begin{theorem}(Serre duality)\label{Serre duality}
Suppose that $M$ is a compact generalized K\"ahler manifold. Then
\begin{enumerate}
\item $H^{p, q}_{\delta_+}(M)\cong (H^{-p, -q}_{\delta_+}(M))^*$, $H^{p, q}_{\delta_-}(M)\cong (H^{-p, -q}_{\delta_-}(M))^*$;
\item $H^{p, q}_{BC, \delta_+\delta_-}(M)\cong (H^{-p, -q}_{BC, \delta_+\delta_-}(M))^*$, $H^{p, q}_{BC, \delta_+\overline{\delta}_-}(M)\cong (H^{-p, -q}_{BC, \delta_+\overline{\delta}_-}(M))^*$.
\end{enumerate}
\end{theorem}

\begin{proof}
\begin{enumerate}
\item Identifying the groups $H^{p, q}_{\delta_+}(M), H^{-p, -q}_{\delta_+}(M)$ with the groups of harmonic forms $\mathscr{H}^{p, q}_{\delta_+}(M)$, $\mathscr{H}^{-p, -q}_{\delta_+}(M)$ respectively.
Note that $\Delta_{\delta_+}\alpha=0$ if and only if $\delta_+\alpha=\delta_+^*\alpha=0$. But on a generalized K\"ahler manifold, $\delta_+^*=-\overline{\delta_+}$. So we have
$\delta_+^*\alpha=0$ if and only if $\delta_+\overline{\alpha}=0$. Hence $\alpha\in \mathscr{H}^{p, q}_{\delta_+}(M)$ if and only if $\overline{\alpha}\in \mathscr{H}^{-p, -q}_{\delta_+}(M)$. Taking $\beta=\overline{\alpha}$, the pairing $\mathscr{H}^{p, q}_{\delta_+}(M)\times \mathscr{H}^{-p, -q}_{\delta_+}(M) \rightarrow \C$ given by $(\alpha, \beta)\mapsto \int_M(\alpha, \star\beta)_{Ch}=h(\alpha, \overline{\beta})$
is nondegenerate.
\item Identifying the groups $H^{p, q}_{BC, \delta_+\delta_-}(M), H^{-p, -q}_{BC, \delta_+\delta_-}(M)$ with the groups of harmonic forms $\mathscr{H}^{p, q}_{BC, \delta_+\delta_-}(M)$, $\mathscr{H}^{-p, -q}_{BC, \delta_+\delta_-}(M)$ respectively. From the proof of Theorem \ref{elliptic}, we have $\Delta_{\delta_+\delta_-}\alpha=0$ if and only if $(\delta_+\delta_-)^*\alpha=\delta_+\alpha=\delta_-\alpha=0$. Since
$\Delta_{\delta_+\delta_-}\alpha=0$ is equivalent to $\Delta_{\overline{\delta_+}\overline{\delta}_-}\overline{\alpha}=0$, using the equalities $\overline{\delta_+}^*=-\delta_+, \overline{\delta_-}^*=-\delta_-$, we
see that $\alpha\in \mathscr{H}^{p, q}_{BC, \delta_+\delta_-}(M)$ if and only if $\overline{\alpha}\in \mathscr{H}^{-p, -q}_{BC, \delta_+\delta_-}(M)$. Hence again by taking $\beta=\overline{\alpha}$, we see that the pairing $\mathscr{H}^{p, q}_{BC, \delta_+\delta_-}(M)\times \mathscr{H}^{-p, -q}_{BC, \delta_+\delta_-}(M) \rightarrow \C$ given by $(\alpha, \beta)\mapsto \int_M(\alpha, \star\beta)_{Ch}$ is nondegenerate.
\end{enumerate}
\end{proof}

\section{Aeppli, Bott-Chern cohomology and $d'd''$-lemma}
For a given double complex $(\mathcal{A}, \dif',\dif'')$, we have the following cohomology groups
$$H^k(\mathcal{A})=\frac{\ker \dif\cap\mathcal{A}^k}{\dif\mathcal{A}^{k-1}}, H_{\dif'}^{p,q}(\mathcal{A})=\frac{\ker \dif'\cap\mathcal{A}^{p,q}}{\dif'\mathcal{A}^{p-1,q}}, H_{\dif''}^{p,q}(\mathcal{A})=\frac{\ker \dif''\cap\mathcal{A}^{p,q}}{\dif''\mathcal{A}^{p,q-1}}$$
By definition, it is clear that
$$H_{\dif''}^{p,q}(\mathcal{A})\cong E_1^{p,q}, H_{\dif'}^{p,q}(\mathcal{A})\cong \overline{E}_1^{p,q}$$

There are other cohomology groups naturally arose from double complexes.
The Bott-Chern and Aeppli cohomology groups are defined as
$$H_{BC}^{p,q}(\mathcal{A})=\frac{\ker \dif'\cap \ker \dif''\cap\mathcal{A}^{p,q}}{\Im \dif'\dif''\cap\mathcal{A}^{p,q}},
H_A^{p,q}(\mathcal{A})=\frac{\ker \dif'\dif''\cap\mathcal{A}^{p,q}}{(\Im \dif'+\Im \dif'')\cap\mathcal{A}^{p,q}}$$

Let $h^k:=\dim_\C H^k(\mathcal{A})$ and
$$h_{\sharp}^{p,q}:=dim_\C H_{\sharp}^{p,q}(\mathcal{A}), h^k_{\sharp}=\sum_{p+q=k}h_{\sharp}^{p,q}, \mbox{ where } \sharp= \dif', \dif'',BC,A$$

Consider the following diagrams
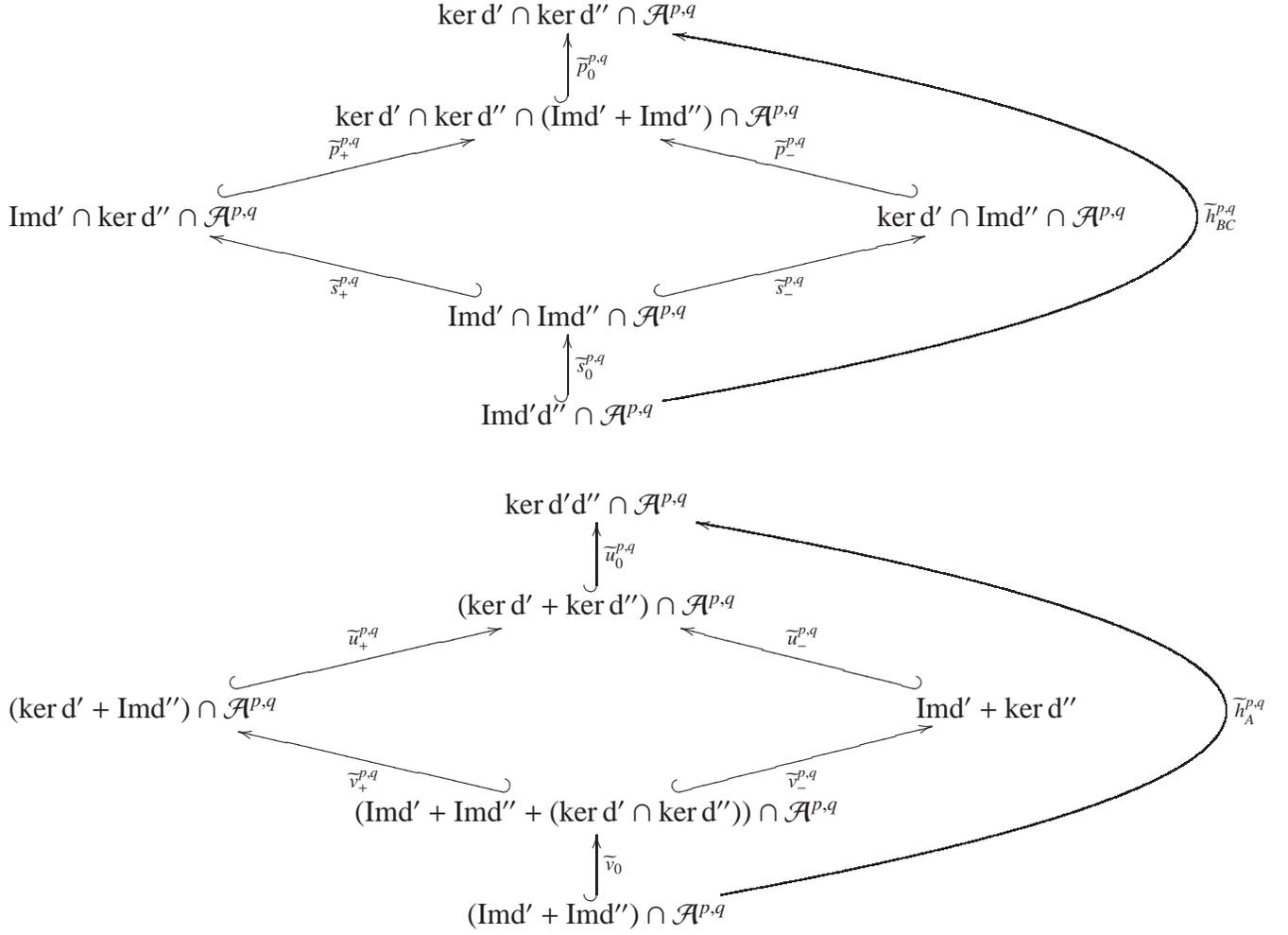
\begin{equation*}
\xymatrix{ & \ker \dif'\cap\ker \dif''\cap \mathcal{A}^{p, q}\\
& \ker \dif'\cap \ker \dif''\cap (\Im \dif'+\Im \dif'')\cap \mathcal{A}^{p, q}\ar@{^(->}[u]_{\widetilde{p}^{p, q}_0}\\
\Im \dif'\cap \ker \dif''\cap \mathcal{A}^{p, q}\ar@{^(->}[ur]^{\widetilde{p}^{p, q}_+} && \ker \dif'\cap \Im \dif''\cap \mathcal{A}^{p, q}\ar@{_(->}[ul]_{\widetilde{p}^{p, q}_-}\\
&\Im \dif'\cap \Im \dif''\cap \mathcal{A}^{p, q}\ar@{_(->}[ul]^{\widetilde{s}^{p, q}_+}\ar@{^(->}[ur]_{\widetilde{s}^{p, q}_-} \\
& \Im \dif'\dif''\cap \mathcal{A}^{p, q}\ar@{^(->}[u]_{\widetilde{s}^{p, q}_0}\ar@/_21pc/[uuuu]_{\widetilde{h}^{p, q}_{BC}}}
\end{equation*}

\begin{equation*}
\xymatrix{ & \ker \dif'\dif''\cap \mathcal{A}^{p, q}\\
& (\ker \dif'+ \ker \dif'')\cap \mathcal{A}^{p, q}\ar@{^(->}[u]_{\widetilde{u}^{p, q}_0}\\
(\ker \dif'+ \Im \dif'')\cap \mathcal{A}^{p, q}\ar@{^(->}[ur]^{\widetilde{u}^{p, q}_+} && \Im \dif'+ \ker \dif''\ar@{_(->}[ul]_{\widetilde{u}^{p, q}_-}\\
& (\Im \dif'+ \Im \dif''+(\ker \dif'\cap\ker \dif''))\cap \mathcal{A}^{p, q}\ar@{_(->}[ul]^{\widetilde{v}^{p, q}_+}\ar@{^(->}[ur]_{\widetilde{v}^{p, q}_-} \\
& (\Im \dif'+\Im \dif'')\cap \mathcal{A}^{p, q}\ar@{^(->}[u]_{\widetilde{v}_0}\ar@/_21pc/[uuuu]_{\widetilde{h}^{p, q}_{A}}}
\end{equation*}

If $\widetilde{a}^{p, q}$ is a homomorphism in the above diagrams, we use $a^{p, q}$ to denote the dimension of coimage.
For instance, $s_+^{p,q}={\rm dim}\dfrac{\Im \dif'\cap \ker \dif''\cap\mathcal{A}^{p,q}}{\Im \dif'\cap \Im \dif''\cap\mathcal{A}^{p,q}}$.

Simple observation gives us some equalities.
\begin{lemma}
For any $p, q$, we have
\begin{enumerate}
\item $p^{p, q}_+=s^{p, q}_-, p^{p, q}_-=s^{p, q}_+$;
\item $h^{p, q}_{BC}=p^{p, q}_0+p^{p, q}_++s^{p, q}_++s^{p, q}_0=p^{p, q}_0+p^{p, q}_-+s^{p, q}_-+s^{p, q}_0$;
\item $u^{p, q}_+=v^{p, q}_-, u^{p, q}_-=v^{p, q}_+$;
\item $h^{p, q}_{A}=u^{p, q}_0+u^{p, q}_++v^{p, q}_++v^{p, q}_0=u^{p, q}_0+u^{p, q}_-+v^{p, q}_-+v^{p, q}_0$;
\item $p^{p, q}_0=v^{p, q}_0$.
\end{enumerate}
\end{lemma}

\begin{lemma}\label{group}
Let $G', G$ be abelian subgroups of some abelian group and $H'\subset G', H\subset G$ be some subgroups. We have
\begin{enumerate}
\item
$\frac{G'}{G'\cap G}\cong\frac{G'+G}{G}$
\item
$(G'+H)\cap (G+H')=H'+H+(G'\cap G)\label{algebra1}$
\item
$(G'\cap H)+ (G\cap H')=G'\cap G\cap (H'+H)\label{algebra2}$
\end{enumerate}
\end{lemma}

\begin{lemma}\label{group2}
If $G'$ is an abelian group and $H<G ,H<H'$ are subgroups of $G'$, the natural map
$\varphi: \dfrac{G}{H}\rightarrow \dfrac{G'}{H'}$
is injective if and only if $G\cap H'=H$, and is surjective if and only if $G'=G+H'$.
\end{lemma}

There are relations involving both diagrams and the cohomology $H_{\dif'}, H_{\dif''}$.
\begin{lemma}
For any $p, q$, there are equalities
$$h^{p, q}_{\dif'}=s^{p, q}_-+v^{p, q}_++v^{p, q}_0, h^{p, q}_{\dif''}=s^{p, q}_++v^{p, q}_-+v^{p, q}_0$$
\end{lemma}

\begin{proof}
Consider the following short exact sequence
$$0\rightarrow \frac{\ker \dif'\cap (\Im \dif'+\Im \dif'')\cap \mathcal{A}^{p, q}}{\Im \dif'\cap \mathcal{A}^{p, q}}\rightarrow H^{p, q}_{\dif'}\rightarrow \frac{\ker \dif'\cap \mathcal{A}^{p, q}}{\ker \dif'\cap(\Im \dif'+\Im \dif'')\cap \mathcal{A}^{p, q}}\rightarrow 0$$

Using Lemma \ref{group}, \ref{group2}, we have
$$\frac{\ker \dif'\cap (\Im \dif'+\Im \dif'')\cap \mathcal{A}^{p, q}}{\Im \dif'\cap \mathcal{A}^{p, q}}\cong \frac{\ker \dif'\cap \Im \dif''\cap \mathcal{A}^{p, q}}{\Im \dif'\cap \Im \dif''\cap \mathcal{A}^{p, q}},$$
$$\frac{\ker \dif'\cap \mathcal{A}^{p, q}}{\ker \dif'\cap(\Im \dif'+\Im \dif'')\cap \mathcal{A}^{p, q}}\cong \frac{(\ker \dif'+\Im \dif'')\cap \mathcal{A}^{p, q}}{(\Im \dif'+\Im \dif'')\cap \mathcal{A}^{p, q}}$$
The equality is given by the alternating sum of dimensions of the short exact sequence and equalities obtained before.
\end{proof}

\begin{corollary}\label{compare}
For any $p, q$, there are equalities
$$h^{p, q}_{BC}+h^{p, q}_{A}=h^{p, q}_{\dif'}+h^{p, q}_{\dif''}+u^{p, q}_0+s^{p, q}_0.$$
In particular, $h^{p, q}_{BC}+h^{p, q}_{A}\geq h^{p, q}_{\dif'}+h^{p, q}_{\dif''}$ for any $p, q$.
\end{corollary}

\begin{definition}\label{dd-lemma}
We say that a double complex $(\mathcal{A}, \dif', \dif'')$ satisfies the $\dif'\dif''$-lemma at $(p,q)$ if
$$\Im \dif'\cap \ker \dif''\cap \mathcal{A}^{p,q}=\ker \dif'\cap \Im \dif''\cap \mathcal{A}^{p,q}=\Im \dif'\dif''\cap \mathcal{A}^{p,q}$$
and $\mathcal{A}$ satisfies the $\dif'\dif''$-lemma if $\mathcal{A}$ satisfies the $\dif'\dif''$-lemma at $(p,q)$ for all $(p,q)$.
\end{definition}

Now we consider following maps induced naturally by inclusions and quotients:
\begin{align*}
&\phi^{p, q}:H^{p, q}_{BC}(\mathcal{A})\stackrel{}{\hookrightarrow}\frac{\ker \dif \cap \mathcal{A}^{p, q}}{\Im \dif'\dif''\cap \mathcal{A}^{p, q}} \stackrel{}{\twoheadrightarrow} H^{p+q}(\mathcal{A}),\\
&\phi^k:=\sum_{p+q=k}\phi^{p, q}:\underset{p+q=k}{\bigoplus}H^{p, q}_{BC}(\mathcal{A})\hookrightarrow \frac{\ker \dif\cap \mathcal{A}^k}{\Im \dif'\dif''\cap \mathcal{A}^k} \twoheadrightarrow H^k(\mathcal{A}), \\
&\psi^k:H^k(\mathcal{A})\stackrel{}{\hookrightarrow} \frac{\ker \dif'\dif'' \cap \mathcal{A}^k}{\Im \dif \cap \mathcal{A}^k} \stackrel{}{\twoheadrightarrow} \underset{p+q=k}{\bigoplus} H^{p, q}_A(\mathcal{A}),\\
&\phi^{p, q}_+:H^{p, q}_{BC}(\mathcal{A})\stackrel{}{\hookrightarrow}\frac{\ker \dif' \cap \mathcal{A}^{p, q}}{\Im \dif'\dif'' \cap \mathcal{A}^{p, q}} \stackrel{}{\twoheadrightarrow} H^{p, q}_{\dif'}(\mathcal{A}), \phi^{p, q}_-:H^{p, q}_{BC}(\mathcal{A})\stackrel{}{\hookrightarrow}\frac{\ker \dif'' \cap \mathcal{A}^{p, q}}{\Im \dif'\dif'' \cap \mathcal{A}^{p, q}} \stackrel{}{\twoheadrightarrow} H^{p, q}_{\dif''}(\mathcal{A}),\\
&\psi^{p, q}_+:H^{p, q}_{\dif'}(\mathcal{A})\stackrel{}{\hookrightarrow} \frac{\ker \dif'\dif'' \cap \mathcal{A}^{p, q}}{\Im \dif' \cap \mathcal{A}^{p, q}} \stackrel{}{\twoheadrightarrow} H^{p, q}_A(\mathcal{A}), \psi^{p, q}_-:H^{p, q}_{\dif'}(\mathcal{A})\stackrel{}{\hookrightarrow} \frac{\ker \dif'\dif'' \cap \mathcal{A}^{p, q}}{\Im \dif'' \cap \mathcal{A}^{p, q}} \stackrel{}{\twoheadrightarrow} H^{p, q}_A(\mathcal{A}).
\end{align*}

The first result in the following was mentioned in the paper (\cite{DGMS}).
\begin{lemma}\label{bc+derham}
\begin{enumerate}
\item
 $\dif'\dif''$-lemma at $\mathcal{A}^k$ holds $\Leftrightarrow$ $\phi^k$ is injective $\Leftrightarrow$ $\psi^{k-1}$ is surjective.
\item
$\phi^k$ is injective $\Rightarrow$ $\phi^{k-1}$ is surjective.
\item
$\psi^k$ is surjective $\Rightarrow \psi^{k+1}$ is injective.
\end{enumerate}
\end{lemma}

\begin{proof}
\begin{enumerate}
\item
Suppose that $\dif'\dif''$-lemma at $\mathcal{A}^k$ holds. Then
$\ker \dif'\cap \ker \dif''\cap \Im \dif \cap \mathcal{A}^k\subseteq (\ker \dif'\cap \Im \dif''+\ker \dif''\cap \Im \dif')\cap \mathcal{A}^k=\Im \dif'\dif''\cap \mathcal{A}^k$ and hence
$\phi^k$ is injective.

Assume that $\phi^k$ is injective.
Let $\alpha\in \ker \dif'\dif''\cap \mathcal{A}^{k-1}$.
Then $\dif\alpha\in \ker \dif'\cap \ker \dif''\cap \Im \dif\cap \mathcal{A}^k=\Im \dif'\dif''\cap \mathcal{A}^k$.
So $\dif\alpha=\dif'\dif''\beta$ for some $\beta\in\mathcal{A}^{k-2}$ and
$\alpha=(\alpha-\dif''\beta)+\dif''\beta\in (\ker \dif+\Im \dif'')\cap \mathcal{A}^{k-1}\subseteq (\ker \dif+\Im \dif'+\Im \dif'')\cap \mathcal{A}^{k-1}$.
This implies that $\psi^{k-1}$ is surjective.

Assume that $\psi^{k-1}$ is surjective.
Consider $\dif'\alpha\in \Im \dif' \cap \ker \dif''\cap \mathcal{A}^{s,k-s}$.
So $\alpha\in \ker \dif'\dif''\cap \mathcal{A}^{k-1}=(\ker \dif+\Im \dif'+\Im \dif'')\cap \mathcal{A}^{k-1}$ and can be decomposed as $\alpha=\bar{\alpha}+\dif'x+\dif''y$ with $\dif \bar{\alpha}=0$.
More precisely, $\bar{\alpha}=\sum \bar{\alpha}_i, x=\sum x_i, y=\sum y_i, \bar{\alpha}_i\in\mathcal{A}^{s+i,k-s-i}, x_i\in\mathcal{A}^{s+i-1,k-s-i}, y_i\in\mathcal{A}^{s+i,k-s-i-1}$.
The $(s,k-s)-$ and $(s+1,k-s-1)-$terms of $\alpha$ are $\alpha=\bar{\alpha}_0+\dif'x_0+\dif''y_0, 0=\alpha_1=\bar{\alpha}_1+\dif'x_1+\dif''y_1$.
So
$\dif'\alpha
=\dif'\bar{\alpha}_0+\dif'\dif''y_0
=-\dif''\bar{\alpha}_1+\dif'\dif''y_0
=\dif''\dif'x_1+\dif'\dif''y_0
=\dif'\dif''(y_0-x_1)
$
Hence $\Im \dif' \cap \ker \dif''\cap \mathcal{A}^{s,k-s}\subseteq \Im \dif'\dif''\cap \mathcal{A}^{s,k-s}$.
Similarly, $\ker \dif' \cap \Im \dif''\cap \mathcal{A}^{s,k-s}\subseteq \Im \dif'\dif''\cap \mathcal{A}^{s,k-s}$.

\item
Assume $\phi^k$ is injective.
Let $\alpha=\sum \alpha_i\in \ker \dif\cap\mathcal{A}^{k-1}, \alpha_i\in \mathcal{A}^{i,k-1-i}$.
By (1), $\dif'\dif''$-lemma at $\mathcal{A}^k$ holds and $\dif'\alpha_i=-\dif''\alpha_{i+1}=\dif'\dif''\beta_i$ for some $\beta_i$.
So $\alpha=(\alpha-\dif\sum \beta_i)+\dif\sum \beta_i \in ((\ker \dif'\cap\ker \dif'')+\Im \dif)\cap \mathcal{A}^{k-1}$.
Hence, $\phi^{k-1}$ is surjective.

\item
Let $\alpha=\dif'\beta'+\dif''\beta''\in \ker \dif \cap (\Im \dif'+ \Im \dif'')\cap\mathcal{A}^{k+1}$.
Then $0=\dif\alpha=\dif''\dif'(\beta'-\beta'')$.
$\psi^k$ is surjective or $\dif'\dif''$-lemma at $\mathcal{A}^{k+1}$ implies that $\dif'(\beta'-\beta'')\in \Im \dif'\cap \ker \dif''\cap\mathcal{A}^{k+1}=\Im \dif'\dif''\cap\mathcal{A}^{k+1}$ and $\dif'(\beta'-\beta'')=\dif'\dif''\gamma$ for some $\gamma\in\mathcal{A}^{k-1}$.
Thus $\alpha=\dif\beta''+\dif'(\beta'-\beta'')=\dif\beta''+\dif'\dif''\gamma=\dif(\beta''+\dif''\gamma)\in \Im \dif\cap\mathcal{A}^{k+1}$ and $\psi^{k+1}$ is injective.
\end{enumerate}
\end{proof}

\begin{lemma}\label{bc+partial}
\begin{enumerate}
\item
 $\dif'\dif''$-lemma at $\mathcal{A}^{p,q}$ holds $\Leftrightarrow$ $\phi_+^{p,q},\phi_-^{p,q}$ are injective $\Leftrightarrow$ $\psi_+^{p-1,q},\psi_-^{p,q-1}$ are surjective.
\item
$\phi^{p, q}_+$ is injective $\Longrightarrow \phi^{p-1, q}_-$ is surjective;
$\phi^{p, q}_-$ is injective $\Longrightarrow \phi^{p, q-1}_+$ is surjective.
\item
$\psi^{p, q}_+$ is surjective $\Longrightarrow \psi^{p+1, q}_-$ is injective;
$\psi^{p, q}_-$ is surjective $\Longrightarrow \psi^{p+1, q}_+$ is injective.
\end{enumerate}
\end{lemma}
\begin{proof}
\begin{enumerate}
\item
Suppose that $\dif'\dif''$-lemma at $\mathcal{A}^{p,q}$ holds. The injectivity of $\phi_+^{p,q},\phi_-^{p,q}$ are clear from definition.

Assume that $\phi_+^{p,q},\phi_-^{p,q}$ are injective.
Let $\alpha\in \ker \dif'\dif''\cap \mathcal{A}^{p-1,q}$.
By injectivity of $\phi_+^{p,q}$, $\dif'\alpha\in \ker \dif'\cap \ker \dif''\cap \Im \dif'\cap \mathcal{A}^{p,q}=\Im \dif'\dif''\cap \mathcal{A}^{p,q}$ and
$\dif\alpha=\dif'\dif''\beta$ for some $\beta\in\mathcal{A}^{p-1,q-1}$.
So $\alpha=(\alpha-\dif''\beta)+\dif''\beta\in (\ker \dif'+\Im \dif'')\cap \mathcal{A}^{p-1,q}$.
Hence $\psi_+^{p-1,q}$ is surjective.
Similarly, $\psi_-^{p,q-1}$ is surjective.

Assume that $\psi_+^{p-1,q},\psi_-^{p,q-1}$ are surjective.
Consider $\dif'\alpha\in \Im \dif' \cap \ker \dif''\cap \mathcal{A}^{p,q}$.
So $\alpha\in \ker \dif'\dif''\cap \mathcal{A}^{p-1,q}=(\ker \dif'+\Im \dif'')\cap \mathcal{A}^{p-1,q}$ and can be decomposed as $\alpha=\bar{\alpha}+\dif''\beta$ with $\dif'\bar{\alpha}=0$.
In particular, $\dif'\alpha=\dif'\dif''\beta$.
Hence $\Im \dif' \cap \ker \dif''\cap \mathcal{A}^{p,q}=\Im \dif'\dif''\cap \mathcal{A}^{p,q}$.
Similarly, we have $\ker \dif' \cap \Im \dif''\cap \mathcal{A}^{p,q}=\Im \dif'\dif''\cap \mathcal{A}^{p,q}$.

\item
Assume $\phi^{p, q}_+$ is injective.
Let $\alpha\in \ker \dif''\cap \mathcal{A}^{p-1, q}$.
Then $\dif'\alpha \in \Im \dif'\cap \ker \dif''\cap \mathcal{A}^{p, q}=\Im \dif'\dif'' \cap \mathcal{A}^{p, q}$ and
$\dif'\alpha=\dif'\dif''\beta$ for some $\beta \in \mathcal{A}^{p-1, q-1}$.
Moreover, $\alpha=(\alpha-\dif''\beta)+\dif''\beta \in ((\ker \dif'\cap \ker \dif'')+\Im \dif'')\cap \mathcal{A}^{p-1, q}$.
So $((\ker \dif'\cap \ker \dif'')+\Im \dif')\cap \mathcal{A}^{p-1, q}=\ker \dif'' \cap \mathcal{A}^{p-1, q}$ and $\phi^{p-1, q}_-$ is surjective.
\item
Assume $\psi^{p, q}_+$ is surjective.
To prove that $\psi^{p+1, q}_-$ is injective, which is equivalent to $\ker \dif''\cap (\Im \dif'+ \Im \dif'')\cap \mathcal{A}^{p+1, q}=\Im \dif''\cap \mathcal{A}^{p+1, q}$, it is enough to show that $\Im \dif'\cap \ker \dif''\cap \mathcal{A}^{p+1, q}\subset \Im \dif''\cap \mathcal{A}^{p+1, q}$.
Let $\dif'\alpha\in \Im \dif'\cap \ker \dif''\cap \mathcal{A}^{p+1, q}$.
Then $\alpha \in \ker \dif'\dif''\cap \mathcal{A}^{p, q}=(\ker \dif'+\Im \dif'') \cap \mathcal{A}^{p, q}$ and
$\alpha=\bar{\alpha}+\dif''\beta$ for some $\bar{\alpha}\in \mathcal{A}^{p, q}, \beta \in \mathcal{A}^{p, q-1}$ with $\dif'\bar{\alpha}=0$.
So $\dif'\alpha=\dif'\bar{\alpha}+\dif'\dif''\beta=\dif''(-\dif'\beta) \in \Im \dif''\cap \mathcal{A}^{p+1, q}$.
\end{enumerate}
\end{proof}

The following result largely generalizes the result obtained by Angella and Tomassini in \cite{AT}.

\begin{theorem}\label{d'd''-lemma}
Given a totally bounded double complex $(\mathcal{A}, \dif', \dif'')$.
Assume that one of the cohomology groups $H_{BC},H_A,H_{\dif'},H_{\dif''},H_\dif$ is finite dimensional. Then
the following statements are equivalent
\begin{enumerate}
\item $\dif'\dif''$-lemma holds;
\item $h^{p, q}_{BC}=h^{p, q}_{\dif'}=h^{p,q}_{\dif''}$ for all $p, q$;
\item $h^{p, q}_A=h^{p, q}_{\dif'}=h^{p,q}_{\dif''}$ for all $p, q$;
\item $h^k_{BC}=b^k$ for all $k$;
\item $h^k_A=b^k$ for all $k$.
\end{enumerate}
\end{theorem}

\begin{proof}
If $\dif'\dif''$-lemma holds, then the other statements hold from Lemma \ref{bc+derham}(1), \ref{bc+partial}(1).

To show that any other statement implies $\dif'\dif''$-lemma, it is enough to show that the corresponding maps are isomorphisms.
Assume $h^{p, q}_{BC}=h^{p, q}_{\dif'}=h^{p,q}_{\dif''}$ for all $p, q$ and $\phi^{s,t}_+$ is not an isomorphism for some $s,t$.
Then $\phi^{s,t}_+$ is neither injective nor surjective.
By Lemma \ref{bc+partial}, $\phi_-^{s,t+1}$ is not injective, and not surjective, either.
By Lemma \ref{bc+partial} again, $\phi_+^{s+1,t+1}$ is not surjective.
Same argument shows that $\phi_+^{s+i,t+i}$ is not an isomorphism for any $i\in \N$.
But totally boundedness of $\mathcal{A}$ guarantees that $\phi_+^{p,q}$ is an isomorphism when $|p|,|q|>>0$, a contradiction.
The proof for $(3)\Rightarrow (1)$ is similar.

Assume $h^k_{BC}=b^k$ for all $k$ and $\phi^t$ is not an isomorphism for some $t$.
Then $\phi^t$ is neither injective nor surjective.
By Lemma \ref{bc+derham}(2), $\phi^{t+1}$ is not injective, and nor is surjective.
Continue this process, we show that $\phi^{i}$ is not an isomorphism for any $i\geq t$, which contradicts the totally boundedness of $\mathcal{A}$.

Assume $h^k_A=b^k$ for all $k$ and $\psi^k$ is not an isomorphism for some $t$.
Then $\psi^t$ is neither injective nor surjective.
By Lemma \ref{bc+derham}(3), $\psi^{t-1}$ is not surjective, and nor is injective.
Continue this process, we show that $\psi^i$ is not an isomorphism for any $i\leq t$, which contradicts the totally boundedness of $\mathcal{A}$.
\end{proof}

\subsection{Applications}
We note that on a compact generalized K\"ahler manifold $M$, we have the equalities (see \cite[Theorem 2.11]{Ca07})
$$\Delta_\dif=2\Delta_{\partial_1}=2\Delta_{\partial_2}=4\Delta_{\delta_+}=4\Delta_{\delta_-}$$
and since these Laplacians are elliptic, by the Hodge decomposition, it is straightforward to see that

\begin{corollary}\cite{Ca07}\label{Kahler dd}
Suppose $M$ is a compact generalized K\"ahler manifold. Then it satisfies the $\partial_1\overline{\partial}_1$-lemma,
$\partial_2\overline{\partial}_2$-lemma and $\delta_i\delta_j$-lemma for $(i, j)=(1, 2), (1, 3), (2, 4), (3, 4)$.
\end{corollary}
More generally, we have
\begin{corollary}\label{dd to BC}
Given a compact bi-generalized Hermitian manifold $M$ that satisfies the $\delta_+\delta_-$-lemma, we have
\begin{enumerate}
\item $H^{p, q}_{\delta_+}(M)= H^{p, q}_{\delta_-}(M)=H^{p, q}_{BC, \delta_+\delta_-}(M)$
for any $p, q\in \Z$,
\item $H^k_{\partial_1}(M)=\bigoplus_{j\in \Z}H^{k, j}_{\delta_+}(M)$ for any $k\in \Z$,
\item $H^k_{\partial_2}(M)=\bigoplus_{j\in \Z}H^{j, k}_{\delta_-}(M)$ for any $k\in \Z$.
\end{enumerate}
\end{corollary}

\begin{proof}
1. and 2. are proved by taking $\mathcal{A}^{p, q}=U^{p+q, p-q}$, $\dif=\partial_1, \dif'=\delta_+, \dif''=\delta_-$. 3. is proved by taking
$\mathcal{A}^{p, q}=U^{p-q, p+q}, \dif=\partial_2, \dif'=\delta_+, \dif''=\overline{\delta}_-$.
Results follow from Theorem \ref{d'd''-lemma}.
\end{proof}

\begin{example}
Let $\J_1, \J_2$ be the generalized complex structures on $S^2$ induced from the complex structure and symplectic structure on $\CP^1$ respectively.
By the equality $H^k_{\overline{\partial}}(S^2)=\underset{p-q=k}{\bigoplus}H^{p, q}_{\overline{\partial}}(S^2)$,
the fact that $(S^2, \J_1, \J_2)$ is a compact generalized K\"ahler manifold, Serre duality and Corollary~\ref{dd to BC}, we get
$$\dim H^{p, q}_{BC, \delta_+\delta_-}(S^2)=\left\{
                                              \begin{array}{ll}
                                                1, & \hbox{ if } (p, q)=(0, 1), (0, -1)\\
                                                0, & \hbox{ otherwise. }
                                              \end{array}
                                            \right.
$$
\end{example}

\section{Some calculations}
\subsection{Generalized Bott-Chern cohomology groups of $\R^2$ and $\T^2$}
Let
\begin{align}
\J_1=\left(
  \begin{array}{cccc}
  a&0 & 0           & b  \\
  0&a& -b &0            \\
           0&-c&-a  &0  \\
  c&0          & 0 &-a  \\
  \end{array}
\right),\ a^2+bc=-1 \label{e5}
\end{align}
and
\begin{align}\J_2=\left(
  \begin{array}{cccc}
  p&q & 0           &0  \\
  r&-p&0 &0            \\
           0&0&-p  &-r \\
 0 &0          & -q &p  \\
  \end{array}
\right),\ p^2+qr=-1 \label{e6}
\end{align}

Then $\J_1, \J_2$ are two translation invariant generalized complex structures on $\R^2$ and the triple $(\R^2, \J_1, \J_2)$ forms a bi-generalized complex manifold.

Nontrivial Bott-Chern cohomology groups $H^{*,*}_{BC,\delta_+\delta_-}(\R^2)$ and $H^{*,*}_{BC,\delta_+\overline{\delta}_-}(\R^2)$ are listed in the following:
\begin{align*}
& H_{BC,\delta_+\delta_-}^{0,1}(\R^2)\cong \Set{f\in C^\infty(\R^2)\otimes \C}{\frac{\partial f}{\partial \bar{\phi}}=0},
& H_{BC,\delta_+\delta_-}^{0,-1}(\R^2)\cong \Set{h\in C^\infty(\R^2)\otimes \C}{\frac{\partial h}{\partial \phi}=0},\\
& H_{BC,\delta_+\delta_-}^{1,0}(\R^2)\cong\frac{C^\infty(\R^2)\otimes \C}{\frac{\partial^2}{\partial \phi\partial\bar{\phi}}(C^\infty(\R^2)\otimes\C)},
& H_{BC,\delta_+\delta_-}^{-1,0}(\R^2)\cong\Set{k\in C^\infty(\R^2)\otimes \C}{\frac{\partial k}{\partial \phi}=\frac{\partial k}{\partial \bar{\phi}}=0},
\end{align*}

\begin{align*}
& H_{BC,\delta_+\overline{\delta}_-}^{0,1}(\R^2)\cong \frac{C^\infty(\R^2)\otimes \C}{\frac{\partial^2}{\partial \phi^2}(C^\infty(\R^2)\otimes \C)},
& H_{BC,\delta_+\overline{\delta}_-}^{1,0}(\R^2)\cong \Set{g\in C^\infty(\R^2)\otimes \C}{\frac{\partial g}{\partial \phi}=0},\\
& H_{BC,\delta_+\overline{\delta}_-}^{0,-1}(\R^2)\cong\Set{h\in C^\infty(\R^2)\otimes \C}{\frac{\partial h}{\partial \phi}=0},
& H_{BC,\delta_+\overline{\delta}_-}^{-1,0}(\R^2)\cong \Set{k\in C^\infty(\R^2)\otimes \C}{\frac{\partial k}{\partial \phi}=0}.
\end{align*}
By Proposition \ref{conjugation}, $H^{*,*}_{BC,\overline{\delta}_+\delta_-}(\R^2)$ and $H^{*,*}_{BC,\overline{\delta}_+\overline{\delta}_-}(\R^2)$ are also obtained.

The computations for the Bott-Chern cohomology groups of $\mathbb{T}^2$ are similar. We can show that $(\mathbb{T}^2,\J_1,\J_2)$ admits a generalized Hermitian structure if and only if $ap=0$.
In this case, Theorem \ref{BCfinite} implies that the Bott-Chern cohomology groups of $(\mathbb{T}^2,\J_1,\J_2)$ are of finite dimension. Thus, we have
$$\dim_{\C}H^{p, q}_{BC, \delta_+\delta_-}(\T^2)=\dim_{\C}H^{p, q}_{BC, \delta_+\overline{\delta}_-}(\T^2)=
\left\{
  \begin{array}{ll}
    \C, & \hbox{ if } (p, q)=(0, 1), (1, 0), (0, -1), (-1, 0);\\
    0, & \hbox{ otherwise. }
  \end{array}
\right.
$$

\subsection{Generalized Bott-Chern cohomology groups of $\R^4$ and $\mathbb{T}^4$}\label{R4}
 Let $\{x_1,y_1,x_2,y_2\}$ be the coordinates on $\R^4$. Two translation invariant generalized complex structures on $\R^4$ are given respectively as
 \[\J_1=\left(
 \begin{array}{cccccccc}
   0 &  0&  0& 0 &  0& 0& 1 &0  \\
   0 &  0&  0& 0 &  0& 0&  0&1  \\
   0 &  0&  0& 0 & -1& 0&  0&0  \\
   0 &  0&  0& 0 &  0&-1&  0&0 \\
   0 &  0&  1& 0 &  0& 0&  0&0  \\
   0 &  0&  0& 1 &  0& 0&  0&0 \\
  -1 &  0&  0& 0 &  0& 0&  0&0  \\
   0 & -1&  0& 0 &  0& 0&  0&0
 \end{array}
 \right),\J_2=\left(\begin{array}{cccccccc}
                                            0&  1&  0&  0&  0& 0 &  0&0  \\
                                           -1&  0&  0&  0&  0& 0 &  0&0 \\
                                            0&  0&  0&  1&  0& 0 &  0&0  \\
                                            0&  0& -1&  0&  0& 0 &  0&0  \\
                                            0&  0&  0&  0&  0& 1 &  0&0  \\
                                            0&  0&  0&  0& -1& 0 &  0&0 \\
                                            0&  0&  0&  0&  0& 0 &  0&1  \\
                                            0&  0&  0&  0&  0& 0 & -1&0
                                         \end{array}
\right)\]
with respect to the ordered basis $\mathscr{B}=\left\{\frac{\partial}{\partial x_1},\frac{\partial}{\partial y_1}, \frac{\partial}{\partial x_2}, \frac{\partial}{\partial y_2},\dif{x_1},\dif{y_1},\dif{x_2},\dif{y_2}\right\}$ of $\T\R^4$.
Notice that $\J_2$ is induced by the complex structures of $\C^2$ with coordinate $z_i=x_i+iy_i, i=1,2$.
But $\J_1$ is  induced by the symplectic structures $\dif x_1\wedge \dif x_2+\dif y_1\wedge \dif y_2$ rather than the K\"ahler form $\dif x_1\wedge \dif y_1+\dif x_2\wedge \dif y_2$.

Let $\mathscr{C}=C^\infty(\C^2)\otimes \C$. We have
\begin{enumerate}
\item $H_{BC,\delta_+\delta_-}^{-2,0}(\R^4)\cong\Set{f\in \mathscr{C}}{\dfrac{\partial f}{\partial {z}_j}=\dfrac{\partial f}{\partial \bar{z}_j}=0\ \text{for}\ j=1,2};$
\item $H_{BC,\delta_+\delta_-}^{2,0}(\R^4)\cong \dfrac{\mathscr{C}}{\Set{\dfrac{1}{4}\left(\dfrac{\partial^2g_1}{\partial z_2\partial\bar z_2}+\dfrac{\partial^2g_2}{\partial z_1\partial\bar z_1}\right)+i\left(\dfrac{\partial^2g_4}{\partial z_1\partial\bar z_2}-\dfrac{\partial^2g_3}{\partial z_2\partial\bar z_1}\right)}{g_i\in \mathscr{C}}}$
\item $H_{BC,\delta_+\delta_-}^{0,-2}(\R^4)\cong\Set{f\in \mathscr{C}}{\dfrac{\partial f}{\partial \bar z_j}=0\ \text{for}\ j=1,2}$;
\item $H_{BC,\delta_+\delta_-}^{0,2}(\R^4)\cong\Set{f\in \mathscr{C}}{\dfrac{\partial f}{\partial z_j}=0\ \text{for}\ j=1,2}$;
\item $H_{BC,\delta_+\delta_-}^{0,0}(\R^4)=\dfrac{\Set{(f_1,f_2,f_3,f_4)\in \mathscr{C}^4}{\dfrac{i}{4}\pd{f_1}{\bar z_2}+\pd{f_3}{\bar z_1}=\dfrac{-i}{4}\pd{f_2}{\bar z_1}+
       \pd{f_4}{\bar z_2}=\dfrac{-i}{4}\pd{f_1}{z_2}+\pd{f_4}{z_1} =\dfrac{i}{4}\pd{f_2}{z_1}+\pd{f_3}{z_2}=0}}{\Set{\left(-4\dfrac{\partial^2 g}{\partial z_1\bar z_1}, -4\dfrac{\partial^2 g}{\partial z_2\bar z_2}, i\dfrac{\partial^2 g}{\partial \bar z_2 \partial z_1}, -i\dfrac{\partial^2 g}{\partial \bar z_1\partial z_2}\right)}{g\in \mathscr{C}}}$
\item $H_{BC,\delta_+\delta_-}^{-1,-1}(\R^4)\cong \Set{(f_1, f_2)\in \mathscr{C}\oplus \mathscr{C}}{\dfrac{\partial f_1}{\partial {z}_2}=\dfrac{\partial f_1}{\partial {z}_1},\dfrac{\partial f_1}{\partial \bar{z}_1}=\dfrac{\partial f_1}{\partial \bar{z}_2}=\dfrac{\partial f_2}{\partial \bar{z}_1}=\dfrac{\partial f_2}{\partial \bar{z}_2}=0}$;
\item $H_{BC,\delta_+\delta_-}^{-1,1}(\R^4)\cong \Set{(f_1, f_2)\in \mathscr{C}\oplus \mathscr{C}}{\dfrac{\partial f_1}{\partial \bar{z}_2}=\dfrac{\partial f_1}{\partial \bar{z}_1},
      \dfrac{\partial f_1}{\partial {z}_1}=\dfrac{\partial f_1}{\partial {z}_2}=\dfrac{\partial f_2}{\partial {z}_1}=\dfrac{\partial f_2}{\partial {z}_2}=0}$;
\item $H_{BC,\delta_+\delta_-}^{1,1}(\R^4)\cong \dfrac{\Set{(f_1, f_2)\in \mathscr{C}\oplus \mathscr{C}}{\pd{f_1}{z_2}-\pd{f_2}{z_1}=0}}{\Set{\left(\dfrac{\partial^2g_1}{\partial z_1\partial\bar z_2}-\dfrac{\partial^2g_2}{\partial z_1\partial\bar z_1}, \dfrac{\partial^2g_1}{\partial z_2\partial\bar z_2}-\dfrac{\partial^2g_2}{\partial z_2\partial \bar z_1}\right)}{g_1,g_2\in \mathscr{C}}}$;
\item $H_{BC,\delta_+\delta_-}^{1,-1}(\R^4)=\cong \dfrac{\Set{(f_1, f_2)\in \mathscr{C}\oplus \mathscr{C}}{\pd{f_1}{\bar z_2}-\pd{f_2}{\bar z_1}=0}}{\Set{\left(\dfrac{\partial^2g_1}{\partial z_2\partial\bar z_1}-\dfrac{\partial^2g_2}{\partial z_1\partial\bar z_1}, \dfrac{\partial^2g_1}{\partial z_2\partial\bar z_2}-\dfrac{\partial^2g_2}{\partial z_1\partial \bar z_2}\right)}{g_1,g_2\in \mathscr{C}}}$.
\end{enumerate}

Let
\[\G=\left(
 \begin{array}{cccccccc}
   0 &  0&  0& 0&  1& 0& 0&0  \\
   0 &  0&  0& 0&  0& 1& 0&0  \\
   0 &  0&  0& 0&  0& 0& 1&0  \\
   0 &  0&  0& 0&  0& 0& 0&1 \\
   1 &  0&  0& 0&  0& 0& 0&0  \\
   0 &  1&  0& 0&  0& 0& 0&0 \\
   0 &  0&  1& 0&  0& 0& 0&0  \\
   0 &  0&  0& 1&  0& 0& 0&0
 \end{array}
 \right).
 \]

We consider $\mathbb{T}^4$ with the generalized complex structures given as the quotients of $\J_1, \J_2$ on $\R^4$ in previous section.
Then $(\mathbb{T}^4,\J_1,\J_2, \G)$ is a compact bi-generalized Hermitian manifold.
By Theorem \ref{BCfinite}, the Bott-Chern cohomology groups of $(\mathbb{T}^4,\J_1,\J_2)$ are finite dimensional. We have

\begin{multicols}{3}
\begin{enumerate}[label=\empty]
\item $H_{BC,\delta_+\delta_-}^{-2,0}(\mathbb{T}^4)\cong \mathbb{C}$,
\item $H_{BC,\delta_+\delta_-}^{2,0}(\mathbb{T}^4)\cong  \mathbb{C}$,
\item $H_{BC,\delta_+\delta_-}^{0,-2}(\mathbb{T}^4)\cong  \mathbb{C}$,
\item $H_{BC,\delta_+\delta_-}^{0,2}(\mathbb{T}^4)\cong  \mathbb{C}$,
\item $H_{BC,\delta_+\delta_-}^{0,0}(\mathbb{T}^4)\cong  \mathbb{C}^4$,
\item $H_{BC,\delta_+\delta_-}^{-1,-1}(\mathbb{T}^4)\cong  \mathbb{C}^2$,
\item $H_{BC,\delta_+\delta_-}^{-1,1}(\mathbb{T}^4)\cong  \mathbb{C}^2$,
\item $H_{BC,\delta_+\delta_-}^{1,1}(\mathbb{T}^4)\cong  \mathbb{C}^2$,
\item $H_{BC,\delta_+\delta_-}^{1,-1}(\mathbb{T}^4)\cong  \mathbb{C}^2$.
\end{enumerate}
\end{multicols}

\bibliographystyle{alpha}

\end{document}